\documentclass[a4paper,11pt]{amsart}

\usepackage{amsfonts,amsmath,mathrsfs,amssymb,amsthm,amscd,latexsym,amstext,amsxtra}

\usepackage[usenames]{color}
\usepackage[all]{xy}
\usepackage{enumerate}
\xyoption{all}
\usepackage{srcltx} 

\usepackage{amsthm, amsmath, amssymb,latexsym}
\usepackage{amsfonts,epsfig,amscd}
\usepackage[]{fontenc}
\usepackage{xy}
\usepackage{enumerate}
\usepackage[]{fontenc}
\usepackage{xy}
\usepackage{xcolor}
\newtheorem{theorem}{Theorem}[section]
\newtheorem{lemma}[theorem]{Lemma}
\newtheorem{corollary}[theorem]{Corollary}

\newtheorem{remark}[theorem]{Remark}
\newtheorem{assumption}[theorem]{Assumption}

\usepackage[colorinlistoftodos]{todonotes}

\newcommand{\nc}{\newcommand}
\nc{\cH}{{\mathcal H}}
\nc{\cA}{{\mathcal A}}
\nc{\cG}{{\mathcal G}}
\nc{\cC}{{\mathcal C}}
\nc{\cO}{{\mathcal O}}
\nc{\cD}{{\mathcal D}}
\nc{\cI}{{\mathcal I}}
\nc{\cB}{{\mathcal B}}
\nc{\cY}{{\mathcal Y}}
\nc{\cK}{{\mathcal K}}
\nc{\cX}{{\mathcal X}}
\nc{\cS}{{\mathcal S}}
\nc{\cE}{{\mathcal E}}
\nc{\cF}{{\mathcal F}}
\nc{\cZ}{{\mathcal Z}}
\nc{\cQ}{{\mathcal Q}}
\nc{\cN}{{\mathcal N}}
\nc{\cP}{{\mathcal P}}
\nc{\cL}{{\mathcal L}}
\nc{\cM}{{\mathcal M}}
\nc{\cT}{{\mathcal T}}
\nc{\cW}{{\mathcal W}}
\nc{\cU}{{\mathcal U}}
\nc{\cJ}{{\mathcal J}}
\nc{\cV}{{\mathcal V}}
\nc{\bH}{{\mathbb H}}
\nc{\bA}{{\mathbb A}}
\nc{\bG}{{\mathbb G}}
\nc{\bC}{{\mathbb C}}
\nc{\bO}{{\mathbb O}}
\nc{\bI}{{\mathbb I}}
\nc{\bB}{{\mathbb B}}
\nc{\bY}{{\mathbb Y}}
\nc{\bK}{{\mathbb K}}
\nc{\bX}{{\mathbb X}}
\nc{\bS}{{\mathbb S}}
\nc{\bE}{{\mathbb E}}
\nc{\bF}{{\mathbb F}}
\nc{\bZ}{{\mathbb Z}}
\nc{\bQ}{{\mathbb Q}}
\nc{\bN}{{\mathbb N}}
\nc{\bP}{{\mathbb P}}
\nc{\bL}{{\mathbb L}}
\nc{\bM}{{\mathbb M}}
\nc{\bT}{{\mathbb T}}
\nc{\bW}{{\mathbb W}}
\nc{\bU}{{\mathbb U}}
\nc{\bD}{{\mathbb D}}
\nc{\bJ}{{\mathbb J}}
\nc{\bV}{{\mathbb V}}
\nc{\bbZ}{{\mathbb Z}}
\nc{\bR}{{\mathbb R}}
\nc{\fr}{{\rightarrow}}
\nc{\co}{{\nabla}}

\nc{\cu}{{\overlineline{\nabla}}}
\nc{\al}{{\alpha}}

\title{Vanishing cohomology on a double cover}
\author{Yongnam Lee and Gian Pietro Pirola}

\address{Department of Mathematical Sciences, KAIST, 291 Daehak-ro, Yuseong-gu, Daejeon 34141, Korea, and
School of Mathematics, KIAS, 85 Hoegiro, Dongdaemun-gu, Seoul 02455, Korea}

\email{ynlee@kaist.ac.kr}

\address{Dipartimento di Matematica, Universit\`a di Pavia via Ferrata 1, 27100 Pavia, Italy}

\email{gianpietro.pirola@unipv.it}


\begin{document}

\begin{abstract}
In this paper, we prove the irreducibility of the monodromy action on the anti-invariant part of the vanishing cohomology on a double cover of a very general element in an ample hypersurface of a complex smooth projective variety branched at an ample divisor.
As an application, we study dominant rational maps from a double cover of a very general surface $S$ of degree$\, \geq 7$ in $\bP^3$ branched at a very general quadric surface to smooth projective surfaces $Z$. Our method combines the classification theory of algebraic surfaces, deformation theory, and Hodge theory.
\end{abstract}
\maketitle



In this paper we continue to study the subfields of rational functions of complex surfaces pursued in \cite{gp}, \cite{lp}, and
\cite{lp2}. Our research has been motivated by the finiteness theorem for dominant rational maps on a variety of general type.  Let $S$ be smooth complex projective variety of general type. The finiteness theorem states that dominant rational maps of finite degree $S\dashrightarrow Z$ to smooth projective varieties of general type, up to birational equivalence of $Z$, form a finite set. The proof of the finiteness theorem was given by Maehara \cite{Ma} under the assumption of boundedness of pluricanonical maps of varieties of general type. This was proved in \cite{HM}, or \cite {Ta}, or \cite{Ts}. Also there are several results estimating the number of rational mappings from a  fixed variety to varieties of general type. See \cite{gp} for references. Of course, there are at least two cases for $Z$: $Z$ is birational to $S$, or $Z$ is rational. It is an interesting question to determine when theses are all the possibilities for $Z$, i.e. $S$ is ``rigid" in the sense of birational dominant maps.

\medskip

The main result in \cite{lp} yields that a very general surface in $\bP^3$ of degree at least 5 is rigid in the sense discussed above.

\begin{theorem} \label{lp} \emph{(=Theorem 1.1 in \cite{lp})}
Let $S\subset \bP^3 $ be a very general smooth complex surface of degree $d>4.$  Let $Z$ be a nonsingular projective surface, and let $f: S\dashrightarrow Z$ be a dominant rational map, then either $Z$ is rational or $Z$ is birational to $S$.
\end{theorem}

The proof was obtained by combining  deformation theory of curves on surfaces,  dimension counts on moduli, and Hodge theoretical methods, especially the Lefschetz theory.  In \cite{lp2}, we treated the dominant rational maps from the product of two very general curves $C\times D$ to nonsingular projective surfaces. There,  we needed to take in account the $H^2$-Hodge structure of \'etale covering of  $C\times D$ together with some improvement on the dimension counts on moduli. See also \cite{ccr} for the complete intersection case.

\medskip

For a very general curve $C$ in the moduli space of curves of genus $g$, Jacobian and Prym varieties of $C$ are simple abelian varieties. This make us to consider the case of double coverings.

\medskip

To explain our results we let $X$ be a smooth complex projective variety of dimension $n\geq 2$ with $h^{n-1, 0}:=h^0(\Omega^{n-1}_X)=0$.  Let $B\subset X$ be a smooth divisor of $X$ and assume that the line bundle $L=\cO_X(B)$ is two divisible $L=M^{\otimes 2}$. Let $\pi: Y\to X$ be the double cover branched at $B$ defined by the square of a line bundle $M$  and let $j$ be the induced involution. Let $H$ be a very ample line bundle on $X$ and $\tilde H=\pi^\ast H$ be its pull back. Let $S$ be a very general element in the linear system $|H|.$  Assume that $S$ is canonical, i.e. the rational canonical map
$k:S\dasharrow |K_S|$ is birational onto its image. Let  $f: S \dasharrow Z$ be a generically finite dominant rational map of degree $\geq 2.$  In Section 10 in \cite{gp}, we proved $p_g(Z):=h^{n,0}=0.$ In fact, from the Lefschetz theory and the irreducibility of the monodromy action on vanishing cycles (cf. Chapters I--III in \cite{v2}) we conclude that  the canonical map $k$ of $S$ factorizes through $f$ gives a contradiction if $p_g(Z)\ne 0$. Using this, and a careful study of deformation theory of curves on surfaces,  dimension counts on moduli, we conclude Theorem~\ref{lp} stated above.

In this paper, we consider the case of double covers, and obtain our desired  application:

\begin{theorem} \emph{(=Theorem~\ref{application})}  Assume that $\tilde H$ and the branch divisor $R$ are very ample. Assume additionally that two rational maps $k: S\dasharrow |K_S|$ and $k': S\dasharrow |K_S(M)|$ are birational onto their images. Let $\tilde S$ be a very general element in $\tilde H$ and let $f:\tilde S \dasharrow Z$ be a dominant rational map where $Z$ is a smooth projective $(n-1)$-fold. Then $Z$ is either birational to $\tilde S$, $S$, or we have  $p_g(Z)=0.$ 
\end{theorem}

The $(n-1)$-th primitive cohomology $P^{n-1}(\tilde S, \bQ)$ has a natural decomposition into the invariant part $P^+$ and the anti-invariant part $P^-$, that respects the monodromy action of suitable pencils. By the ampleness of the ramification divisor $R$, the $P^-$ turns out to be  the anti-invariant part of the vanishing cohomology. Then we show by means of the monodromy action that  the anti-invariant part of vanishing cohomology is irreducible. This is the content of Theorem~\ref{irreducible}. Here we only deal with cases that are of interest to us.  It would be very useful to develop some complete equivariant Lefschetz theory.

As an application we obtain the following theorem by means of the proof of theorem 1.1 in \cite{lp} and Theorem~\ref{application}:

\begin{theorem} \emph{(=Theorem~\ref{surface})} Let $X=\bP^3$ and let $H=\cO_{\bP^3}(d)$ where $d\ge 7$. Let $\pi: Y\to X$ be a double cover branched at a very general quadric surface.  Suppose there is a dominant rational map $f:\tilde S \dasharrow Z$ where $Z$ is a smooth projective surface. Then $Z$ is birational to $\tilde S$, $S$, or $\bP^2.$
\end{theorem}

As far as we know this gives the first examples of fields of transcendence degree $2$ of non-ruled surfaces that contain only one proper non-rational subfield of transcendence degree $2$. We expect our result extends to a Galois cover with some group $G$, i.e. for a suitable general member we expect to get only intermediary steps that are controlled by $G$.

\medskip

Picard-Lefschetz theory yields that the monodromy action on the $(n-1)$-th vanishing cohomology of a smooth hyperplane section of the variety $X$ of dimension $n$ is irreducible. In Section 1, we extend this irreducibility of the monodromy action on the anti-invariant part of the $(n-1)$-th vanishing cohomology of a smooth hyperplane section of $Y$ double cover branched over a smooth ample divisor. In Section 2, we prove the above theorems as applications.

\bigskip

In this paper we work on the field of complex numbers.

\subsection*{Acknowledgements}
This work was initiated when the second named author visited KAIST supported by KAIST CMC and the first named author visited University of Pavia supported by INdAM (GNSAGA). Both authors thank KIAS and INdAM. The first named author would like to thank University Pavia for their hospitality and Lidea Stoppino for useful discussions during his visit. He was partially supported by Basic Science Research Program through the NRF of Korea funded by the Ministry of Education (2016930170). The second named author is partially supported by INdAM (GNSAGA); PRIN 2017 \emph{``Moduli spaces and Lie theory''} and by  (MIUR): Dipartimenti di Eccellenza Program (2018-2022) - Dept. of Math. Univ. of Pavia. Finally, we thank the referee for several useful suggestions and remarks.

\medskip

\section{Double coverings and the irreducibility of the monodromy action}
 
In this section, after introducing the double coverings (in \ref{double}), we will show an irreducible result of the anti-invariant cohomology (\ref{irreducible}) for them. This will be our main technical point.

\subsection{Double coverings}\label{double}
Let $X$ be a smooth projective variety of dimension $n>1$. Let $B\subset X$ be a smooth divisor of $X$ and assume that the line bundle $L=\cO_X(B)$ is
two divisible. This means that $L=M^{\otimes 2}$ where $M$ is a line bundle of $X.$
Let $\pi: Y\to X$ be the double cover branched at $B$ and let $j$ be the induced involution.
The variety  $Y$ is smooth and the ramification divisor $R$ is isomorphic to $B.$ Let $H$ be a very ample line bundle on $X$ and
$\tilde H=\pi^\ast H$ be its pullback. 

The following lemma might be well known. Since we do not find a suitable reference, we give the proof here.

\begin{lemma}\label{ample}
The line bundle $\tilde H$ is very ample on $Y$ if and only if $H\otimes M^{-1}$ is generated by global sections on $X.$ \end{lemma}
\begin{proof}
Assume that $H\otimes M^{-1}$ is generated by global sections. We have the identification
\[H^0(Y,\tilde H)=H^0(X,\pi_\ast \tilde H)=H^0(X,H)\oplus H^0(X,H\otimes M^{-1})=W^{+}\oplus W^{-}\]
where the sign corresponds to the positive and the negative eigenvalue induced by the involution $j^\ast$ induced by $j.$
 We show now that if $p\neq j(p),$  then points $p$ and $j(p)$ are separated by the global sections of $\tilde H.$  In fact we can find sections $s^+\in W^+$ and $s^-\in W^-$ such that $s^+(p)=s^-(p)\neq 0.$  It follows that $s^-(jp)=-s^+(jp).$
Therefore $s=s^++s^-$ vanishes on $j(p)$ but not on $p.$ Next we show that $\tilde H$ separates the tangents at the point $p\in R$  of the ramification divisor. Since the line bundle $H$ is very ample this will complete the result.
We have to show the surjectivity of the map $\psi$ induced by derivation:
$$  \psi : H^0(Y,\tilde H)\to \tilde H\otimes T^\ast_{Y,p}.$$
For a nonzero vector $v\in T_{Y,p},$ we have to find a section $s\in H^0(Y,\tilde H)$ such that $v\cdot\psi (s)\neq 0.$
The differential of $j$ gives the eigenvector decomposition  $T_{Y,p}=T^+\oplus T^-.$ Choose local coordinates $\{U,x_i\}$ such that $x_i(p)=0,$
$R\cap U=\{x_1=0\}$  and give the linearization of $j:$
$(x_1,x_2,\dots, x_n)\mapsto (-x_1,x_2,\dots, x_n).$ The map $\pi$ in these coordinates becomes $$(x_1,x_2,\dots, x_n)\mapsto(x_1^2=y,x_2,\dots, x_n),$$  and $y=0$ gives  the equation of $B\cap U.$
We see that $\dim T^-=1.$
Write $v=v^++v^-.$ If $v^+\neq 0$ then we can find a section $s=s^+$ such that
$v^+\cdot s\neq 0$ and then $v\cdot s^+=v^+\cdot s \neq 0.$ Therefore we can suppose that $v=v^{-}.$  We need to find a section $s\in W^-$ that vanishes at $p$ of order $1$ in the $v$-direction.  In our coordinates
we may assume  $v=\frac{\partial}{\partial x_1}.$  Taking a trivialization of $H$ on $U,$ the restriction of the section $s\in W^-\subset H^0(X,H)$ becomes a regular function
$f=f(s)$ such that $f(-x_1,x_2,\dots, x_n)=-f(x_1,x_2,\dots, x_n).$  It follows that
$f=x_1g(x_1,x_2,\dots, x_n)$ where $g$ is $j$-invariant. Under the identification  $ W^-= H^0(X,H\otimes M^{-1})$   the local expression of $s$ in this trivialization  is given by $g.$
Therefore
$$ v\cdot s=\frac{\partial f}{\partial x_1}(0,\dots,0)=g(0,\dots,0).$$ Now  $g(0)\neq 0$ holds $\iff s$ is not in the kernel of the restriction map $$H^0(X,H\otimes M^{-1})\to H^0(X,H\otimes M^{-1})_p.$$  The existence of such an $s$ is equivalent to the surjectivity of
$H^0(X,H\otimes M^{-1})\to H^0(X,H\otimes M^{-1})_p.$
The converse is clear.
\end{proof}

\subsection{Lefschetz pencils}

With the previous notation we assume that $H$ and $\tilde H$ are very ample. Let $\ell \subset |H|$ be a pencil of global sections of $H $ i.e., $\ell=\bP(V)$ where $V$ is a two-dimensional subspace of the vector space $H^0(X,H)$. Let
$ \tilde{\ell}$ be the pull back of  $\ell.$
We also assume
\begin{enumerate}
\item the pencil $\{H_t\}_{\in \ell}$ is a Lefschetz pencil of $X;$
\item the restriction of $\ell $ to $B$ : $\{B\cap H_t\}_{\in \ell}$ is a Lefschetz pencil $\ell_B$ of $B.$
\end{enumerate}

Our definition of a Lefschetz pencil is the classical one: any singular fiber has only one singular node (Chapter $2$ in \cite{v2}).  If  the dimension of $B$ is one, we ask for a simple ramification of the map  $B\to \bP^1.$

By blowing-up the base loci of the pencils, we obtain $\tilde B\subset \tilde X$  and $\tilde Y$, and we get
  the fibrations
\begin{equation}\label{pX}
h: \tilde X \to \bP^1
\end {equation}
\begin{equation}\label{pB}
h_B: \tilde B \to \bP^1.
\end {equation}

We also have
\begin{equation}\label{pY}
g: \tilde Y \to \bP^1
\end {equation}
and a two-to-one map that ramifies on $\tilde B:$

\[\xymatrix@C-1pc{
  \tilde Y \ar[rr]^{\tilde \pi} \ar[dr]_{g} &&  \tilde X \ar[ld]^{h} \\
  & \bP^1.
}\]

With the exception (that we may exclude from now on)
$X=\bP^2,$ $B$ a conic and $H=\cO_{\bP^2}(1),$
any divisors in the pencil  ${\ell}$ intersect  transversally $B$ in at least one point,  this gives that all the divisors in
$\tilde \ell$ are irreducible. Then ${g}$ is a connected fibration, and all fibers are irreducible and have at most nodes. 
 
Nevertheless $ \tilde{\ell}$ is not a Lefschetz pencil on $\tilde Y$ because it has a fiber with two nodes.

The singular fibers of $g$ are of two types
\begin {enumerate}
\item [I)] the inverse image $Y_{s'}=\tilde\pi^{-1}(X_{s'})$ of a singular divisor $X_{s'}\in {\ell }$;
\item [II)] the inverse image   $Y_{s''}=\tilde\pi^{-1}(X_{s''})$  where $B\cap X_{s''}=X_{B,s''}\in\ell_B$ is singular.

\end{enumerate}

The set  $S\subset \bP^1$ of the critical values of ${h},$ has a decomposition $$S=S'\cup S''$$
where $S'$ are the critical
values of $h$ and $S''$ the critical values of $h_B.$
From now on  we will assume the following:
\begin{assumption} The intersection of $S'$ and $S''$ is empty:
 $$S'\cap S''=\emptyset.$$ \label{ass}
 \end{assumption}

\begin{remark} The general pencil of hyperplanes ${\ell}$ of $H$ satisfies  our assumption. In fact by the biduality theorem the intersection $X^\ast \cap B^\ast$ of the dual varieties  of $X\subset \bP^N=|H|$ and of $B\subset {\bP^N}$   has codimension $\geq 2$ in ${\bP^N}^\ast.$ Therefore a general pencil in ${\bP^N}^\ast$ does not meet in $X^\ast \cap B^\ast$. We thank the referee for pointing out that this hypothesis is necessary for our purposes. \end{remark}

From (\ref{ass}) we see that in the case I) $\tilde Y_{s'}$ has two singular  nodal points and in the case II) $Y_{s''}$ is  simply tangent to $B$ in one point $p\in S'$ and
$Y_{s''}$ has only one nodal singularity. 

\begin{lemma}
Choose a point $p\in V:=\bP^1\setminus S$. Then we have a decomposition of $(n-1)$-th primitive cohomology $P^{n-1}(Y_p,\bQ)$ into the invariant and anti-invariant part:
$P^{n-1}(Y_p,\bQ)=P^{+}+P^{-}.$
\end{lemma}

\begin{proof}
Set  $U= {g}^{-1}(V),$
then by restriction we have a smooth fibration
$${g}_U: U\to V.$$
We consider the local system $R^{n-1}{g_U}_\ast \bQ $ defined over $V.$
By fixing  a point $p\in V,$  this  local system is equivalent
to  the monodromy action of the fundamental group $\pi_1(V,p)$ on $H^{n-1}(Y_p,\bQ),$ where $ Y_p= {g}^{-1}(p).$
The involution $j$  gives a decomposition
$$R^{n-1}{g}_{U\ast} \bQ=R^+\oplus R^-$$
and  $H^{n-1}(Y_p,\bQ)=H^+\oplus H^-$
into the invariant and the anti-invariant parts.
We get that $H^+$ is isomorphic to $ H^{n-1}(X_p,\bQ)$ where
$X_p={h}^{-1}(p)$ is the fiber of $h.$  Let $P(X,\bQ)\subset H^{n-1}(X,\bQ)$ (resp. $P(Y,\bQ)\subset H^{n-1}(Y,\bQ)$)
 be the primitive cohomology which respect to
$H$ (resp. $\tilde H).$ We can also define the subsystems $P^+\subset R^+$ and $P^-\subset R^-$ that are given as $\pi_1(V,p)$ modules on
$P^{n-1}(Y_p,\bQ).$
 Then we can decompose $P^{n-1}(Y_p,\bQ)$ into the invariant and anti-invariant part:
$P^{n-1}(Y_p,\bQ)=P^{+}+P^{-}.$
\end{proof}

Consider again  $S'\cup S''=S,$ that we separate by open disks $D',D'',$ $S'\subset D'\subset \bP^1$ and $S''\subset D''\subset \bP^1$ such that $ D'\cap D''=\emptyset$  and such that the base point $p$ is in the closure of the disk\ $p\in \bar D' \cap\bar D''.$  Fix generators of $\pi_1(V,p)$ corresponding to loops $\gamma_s, s\in S$  in such a way that $\gamma_{s'}(t)\in \bar D'$ and $\gamma_{s''}(t)\in \bar D''$ for
 $t\in [0,1].$ Then we define the free groups $G'$ and $G''$ generated by the loops around  points of $S'$ and of $S''$, respectively:
$G'=<[\gamma_{s'}]>_{\{ s'\in S'\}}$ and $G''=<[\gamma_{s''}]> _{\{ s''\in S''\}}.$ We have that
$$\pi_1(V,p)= G'\ast G''/\alpha,$$
where $\ast$ stands for the free product and $\alpha$ is the relation given by a suitable product of the loops, homotopically equivalent to the boundary of the above disks.

\bigskip

Let $i:X_p\to  X$ be the inclusion. Then we have by Lefschetz theory
$$P: = P^{n-1}(X_p,\bQ)=i^\ast P^{n-1}(X)\oplus H^{n-1}(X_p)_{van},$$
where $H^{n-1}(X_p)_{van}$ is the kernel of $$i_\ast: H^{n-1}(X_p) \to H^{n+1}(X_p).$$
We have that $H^{n-1}(X_p)_{van}$ is an irreducible  $G'$-module generated by  the vanishing cycles $\delta_{s'}$ where $s'\in S'$  for the fibration
$ \tilde X\to \bP^1,$  And the $\delta_{s'}$  are all conjugate (see Chapter 3 in \cite {v2}) by $G'$.
We have for  $s'\in S'$  $$\tilde\pi(\delta_{s'})=\alpha_{1s'}+\alpha_{2s'}$$ where $\alpha_{is'}$ $i=1,2$ are the vanishing cycles of the two nodes over a point $s'\in S'.$ Since the $\delta_{s'}$ are  conjugated by $G'$ it follows that, up to a sign, the cycles $\beta_{s'}= \alpha_{1s'}-\alpha_{2s'}\in P^-,$  are  conjugated  by the group $G'$. In a similar way the cycles $\beta_{s''}\in P^-,$ where $s''\in S''$ the vanishing cycle  corresponding to the points $s''\in S''$
are  conjugated by the group $G''.$

\begin{lemma}
Let  $H_{S'}\subset P^-$ be the subspace generated   by $\{\beta_{s'}\}_{\{ s'\in S'\}},$ and   $H_{S''}\subset P^-$  be the subspace  generated by 
$\{\beta_{s''}\}_{\{ s''\in S''\}}.$ Then we have $H_{S'}+H_{S''}=H_{van}^-$ where $H_{van}^-=H^{n-1}(Y_p)_{van}^{-}.$
\end{lemma}

\begin{proof}
We  let $\bG :={\rm Gr}_1(|\tilde H|)$ be the Grassmannian variety parametrized lines of the linear system $|\tilde H|$. Let $\tilde\ell\in \bG.$
Let $W\subset \bG$ be the subset parametrized Lefschetz pencil in $|\tilde H|$.

It is classical (see for instance Chapter 2 in \cite{v2}) that $W$ is a nonempty Zariski open subset of $\bG.$
Let $\Delta$ be the complex unit disk, we can find a curve $\rho: \Delta\to  \bG$, such that $\rho(0)=\tilde \ell$ and $\rho(t)=\tilde\ell_t\in W$ for $t\neq 0.$ We consider the singular divisors in the pencil  $\tilde\ell_t$ and its singular points $S(t)$. Then
for any $s'\in S'$ we define then two curves $s'_1(t)$ and $s'_2(t)$ in $S(t)$ for $t\in \Delta$ such that $s'_1(0)=s'_2(0)$ and for any point $s''\in S''$
a curve $s''(t)\in S(t)$ such that $s''(0)=s''.$
We have then by continuity we may assume
\begin{enumerate}
\item  the vanishing cycle of $s'_1(t)$ is $\alpha_{1s'};$
\item  the vanishing cycle of $s'_2(t)$ is $\alpha_{2s'};$
\item   the vanishing cycle of $s'' (t)$ is $\beta_{s''}.$
\end{enumerate}

Applying Lefschetz theory to $\tilde\ell_t$ (cf. \cite{v2}) we see that $\{\alpha_{1s'},\alpha_{2s'}\}_{s'\in S'}\cup \{\beta_{s''}\}_{s''\in S''}$ generates
the cohomology of  $H^{n-1}(Y_p)_{van}$. And $H^{n-1}(Y_p, \bQ)$ is $$H^{n-1}(Y_p)_{van}\oplus i'^\ast H^{n-1}Y $$
where $ i': Y_p\to Y$ is the inclusion.
It follows that  $\{\alpha_{1s'}-\alpha_{2s'}\}_{s'\in S'}\cup \{\beta_{s''}\}_{s''\in S''}$ generates the vanishing part of the anti-ivariant part of the primitive cohomology.
\end{proof}

Another application of Lefschetz theory gives:

\begin{theorem} \label{irreducible} If $\tilde{H}$ is very ample then the  action of the monodromy of $\pi_1(V,p)$ on   $H_{van}^-$ is irreducible. \end{theorem}

\begin{proof}
Let  $ F\subset H_{van}^-$ be a sub-local system. We have to show  that either $F=0$ or $F=H_{van}^-.$ Let $F'$ be a sub-local system orthogonal to $F$:
$$  F'=\{v \in H_{van}^- :<v,w>=0,\forall w\in F\}.$$
 Note that $F'=0\iff F =H_{van}^-$ and   $F=0\iff F' =H_{van}^-$ since the polarization is non-degenerate on $H_{van}^-.$

Now we have that for all $s\in S$  either $\beta_s\in F$ or $\beta_s\in F'.$
To see this, we take, for instance, $s'\in S'$ and any $v\in F$. 
Let $T $ be the monodromy around $s'$, then we must have $T(v)\in F$. The monodromy map can be computed by means of the 
Picard-Lefschetz formula, and it gives (cf.  Theorem 3.16, Chapter 3 in \cite{v2})
$$T(v)=v+<v, \alpha_{1s'}> \alpha_{1s'}+ <v, \alpha_{2s'}>\alpha_{2s'}.$$

As  $v\in F\subset P^-$  then  we have also $$0=<v,\alpha_{1s'}+\alpha_{2s'}>=<v,\alpha_{1s'}>+<v,\alpha_{2s'}>,$$ therefore
$$ T(v)=v+ <v, \alpha_{1s'}> \alpha_{1s'}- <v,\alpha_{1s'}>\alpha_{2s'}=v+<v, \alpha_{1s'}> (\alpha_{1s'}-\alpha_{2s'})$$
$$v+ \frac{1}{2}<v, \alpha_{1s'}-  \alpha_{2s'}>  (\alpha_{1s'}-\alpha_{2s'})= v+ \frac{1}{2}<v, \beta_{s'}>  \beta_{s'}.$$

Now  $v\in F$  and  $T(v)\in F$ gives that  $T(v)-v\in F,$ that is $$<v, \beta_{s'}> \beta_{s'}\in F.$$ Then either 
 $\beta_{s'}\in F$ or $<v, \beta_{s'}> =0,$ for all $v\in F,$ that is $\beta_{s'}\in F'.$ A similar computation applies to the $\beta_{s''}, s''\in S''.$

\medskip

Interchanging $F$ with $F'$ if necessary we may assume  that there is $s'\in S'$ such that $\beta_{s'}\in F.$
Since all the $\beta_{s'}$ are conjugate by $G',$ we obtain then $H_{S'}\subset F.$
If $F$ contains an element $\beta_{s''}\ s''\in S''$ the same argument shows that $F\supset H_{S''}$ therefore
 $F\supset H_{S'}+H_{S''}=H_{van}^-$ and the proof is complete. If we  assume by contradiction that is not the case  we will have $F=H_{S'}$ and $F'\supset H_{S''}.$ In particular  this implies  for all $s'\in S'$ and $s''\in S'':$
$$<\beta_{s'},\beta_{s''}>=< \alpha_{1s'}-\alpha_{2s'},\beta_{s''}>=0.$$

We have also $$< \alpha_{1s'}+\alpha_{2s'},\beta_{s''}>=0$$ since $\beta_{s''} \in P^{-}$ and  $\alpha_{1s'}+\alpha_{2s'}\in P^+.$

That is $<\alpha_{1s'} ,\beta_{s''}>=   <\alpha_{2s'},\beta_{s''}>=0,$ but in this case we will have that $H_{S''}=F'$ is invariant  by the monodromy around all the critical points $s'_1(t), s'_2(t)$,  and $s''(t) $  of the pencil $\tilde\ell_t$, for $t\neq 0.$ This gives a  contradiction with Lefschetz irreducibility theorem (cf. \cite{v2}) since $\tilde\ell_t \in W$ is a Lefschetz pencil.
\end{proof}

\begin{corollary}\label{inj}
With the previous notation we assume that the ramification divisor $R$ is  ample. Then $P^-=H_{van}^-$  and therefore it is irreducible.
\end{corollary}
\begin{proof}
As $R$ is ample the map $H^{n-1}(Y)\to H^{n-1}(R) $ is injective. Since the cohomology $H^{n-1}(R)$ is $j$ invariant  it follows that $H^{n-1}(Y)^-=0. $
Then it follows that $PH^{n-1}(Y_p)^-=H^{n-1}(Y_p)^{-}_{van}.$
\end{proof}

We note that the ampleness of $B$ is equivalent to the ampleness of $R$ by Lemma~\ref{ample}.

\medskip

\section{Applications}

\subsection{Geometric genus of very general double coverings}\label{app}

We recall our notation. Let $X$ be a smooth projective $n$-fold and $H$ be a very ample divisor on $X$. 
Let $\pi: Y\to X$ be a double cover ramified over $R\subset Y$, branched at $B\subset X,$ and  $M^{\otimes 2}=\cO_X(B)$. Let $S$ be a very general element of $|H|.$ Let $K_S= \cO_S(K_X+H)$ be the canonical bundle of $S$ and set
$K_S(M)= \cO_S(K_X+H+M).$ We consider the canonical rational map
$ k: S\dasharrow |K_S|$ and $k': S\dasharrow |K_S(M)|$. We finally set $\tilde S=\pi^{-1}(S).$

\begin{assumption}
\emph{We assume additionally:
\begin{enumerate}
\label{xx}
\item $h^0(\Omega^{n-1}_X)=0;$
 \item  $B$ is smooth and very ample (or ample and base points free);
\item $k$ and $k'$ are birational onto its image.
\end{enumerate}}
\end{assumption}

It follows immediately that $\tilde k: \tilde S\dasharrow  |K_{\tilde S}|$ is birational onto its image (if it factorizes through $\tilde S\to S\dasharrow |K_S|$ then the anti-invariant part must be trivial).

Moreover $h^0(\Omega^{n-1}_Y)=0$: In fact from Corollary~\ref{inj}, $H^{n-1}(Y,\bQ)^-=0$. And we have
$$H^0(\Omega^{n-1}_Y)= H^0(\Omega^{n-1}_Y)^+\oplus H^0(\Omega^{n-1}_Y)^- = H^0(\Omega^{n-1}_X)
\oplus H^0(\Omega^{n-1}_Y)^- =0$$ since $H^0(\Omega^{n-1}_X)=0$ and by the Hodge decomposition
$$H^0(\Omega^{n-1}_Y)^-\subset H^{n-1}(Y,\bQ)^-\otimes \bC\cong H^{n-1}(Y,\bQ)^-\otimes \bC =0.$$

\begin{theorem} \label{application} Let $f:\tilde S \dasharrow Z$ be a dominant rational map where $Z$ is a smooth projective $(n-1)$-fold.
Then $Z$ is either birational to $\tilde S$, $S$, or we have $p_g(Z)=0.$ \end{theorem}

Before giving the proof we recall a standard notion.
Let $V$ be a smooth projective $n$-fold with $n\ge 2$. Let $H^n_V$ be the Hodge structure
to $H^n(V,\bZ).$ The transcendental Hodge structure $\bT_V$ of $V$  is the smallest sub-hodge structure of $H^n_V$ such that $\bT_V^{n, 0}=H^{n,0}(V)$.  We mention that $p_g(V)=0$ is equivalent to $\bT_V=0$. We recall that  $\bT_V$ is a birational invariant and moreover if $f:V\dasharrow W$ is a dominant rational  of finite degree we have an injective Hodge-structure map $f^\ast : \bT_{W}\to \bT_{V}$ (this can be seen by resolving the indeterminacy of $f$). Let $\tilde f:\tilde V\to \tilde W$ be the map after resolving the map $f$. Let $H$ be a very ample divisor of $V$ and $S$ be a very general element of $|H|.$ Suppose $h^0(\Omega^{n-1}_V)=0.$  It follows then that
$H^{n-1}(S)_{van}=\bT_S$ : it contains $H^{n-1,0}(S)$ since $H^{n-1,0}(V)=0$, and it is irreducible by Lefschetz theory.

\begin{proof}  We set $\bT=\bT_{\tilde S}^{n-1}.$  Then $\bT$ is decomposed into 
$\bT^{+}\oplus \bT^{-}$. Since $h^0(\Omega^{n-1}_X)=h^0(\Omega^{n-1}_Y)=0$ we get that $\bT^+=H^{n-1}(S)_{van}$ and
$\bT^{-}=H^{n-1}(\tilde S)^-_{van}.$  Then
under our hypothesis they are both irreducible (\ref{inj}). Assume by contradiction that $p_g(Z)\neq 0$ then the transcendental Hodge structure $ \bT_Z$ is not zero.  We get $f^\ast \bT_Z\subset  \bT^{+}\oplus \bT^{-}$. And then $f$ factorizes through $k$, $k'$, or $\tilde k$ accordingly $f^\ast \bT_Z= \bT^{+}$,
$f^\ast \bT_Z= \bT^{-}$, or $ f^\ast \bT_Z= \bT^{+}\oplus \bT^{-}.$ Since the maps are all birational it proves our theorem.
\end{proof}

We repeat the double covering construction. For each $i=0,1,2,\ldots$, $X^{i}$ is smooth projective $n$-fold, $\pi_i: X^{i+1}\to X^{i}$  a two-to-one 
map ramified on $R^{i}\subset X^{i+1}$ and branched at $B^{i}\subset X^{i}$. Let $M_{i}^{\otimes 2}=\cO_{X^{i}}(B^{i})$, $H^{i}$ be a very ample divisor on $X^{i},$ and 
$S^{i}$ be a very general element of $|H^{i}|.$ Let $K_{S^{i}}= \cO_{S^{i}}(K_{X^{i}}+H^i)$ be the canonical bundle of $S^i$ and set
$K_{S^{i}}(M_i)= \cO_{S^i}(K_{X^i}+H^i+M_i).$ We consider the canonical rational map
$ k_i: S^i\dasharrow |K_{S^i}|$ and $k_i': S^i\dasharrow |K_{S^i}(M_i)|$. We set $S^{i+1}=\pi_i^{-1}(S^i).$

Let $X_0=X, X_1=Y, M_0=M, H^0=H, B^0=B, R^0=R, S^0=S, S^1=\tilde S$ in Assmption~\ref{xx}.
\begin{assumption}
\emph{We assume additionally:
\begin{enumerate}
\item $h^0(\Omega^{n-1}_{X^0})=0;$
\item Branch divisors $B^i$ are smooth and very ample (or ample and base points free) for all $i$;
\item $k_i$ and $k_i'$ are birational onto their images for all $i$.
\end{enumerate}}
\end{assumption}

By the same argument in Theorem~\ref{application}, we get the following.

\begin{corollary}
Let $f^i:S^i \dasharrow Z$ be a dominant rational map where $Z$ is a smooth projective $(n-1)$-fold.
Assume ${\rm deg} f^i>1.$Then either $Z$ is birational to one of $S^j$ for $j=0,1, \ldots, i-1$, or  $p_g(Z)=0.$
\end{corollary}

\begin{corollary} Let $X=\bP^2$ and let $H=\cO_{\bP^2}(d)$ where $d\ge 4$. Set $M=\cO_{\bP^2}(a)$ with $1\leq a\leq d-1$. Let $\pi: Y\to X$ be a double cover branched at a very general element $B\in |M^{\otimes 2}|$.    Let $C$ be a very general element in $|H|$ and let $\tilde C=\pi^{-1}(C)$. Suppose there is a finite  map $f:\tilde C \to Z$ where $Z$ is a smooth projective curve. Then $Z$ is isomorphic either to $\tilde C$, $C$, or $Z=\bP^1.$
\end{corollary}

\begin{proof} Under our hypothesis, $H\otimes M^{-1}=\cO_{\bP^2}(d-a)$, $K_C=\cO_C(d-3)$, and $K_C+M=\cO_C(d+a-3)$ are very ample. Moreover, we have $H^1(\bP^2, \bQ)=0$.
\end{proof}

\begin{corollary}\label{van}
Let $X=\bP^3$ and let $H=\cO_{\bP^3}(d)$ where $d\ge 5$.  Let $M=\cO_{\bP^3}(a)$ with $1\leq a\leq d-1$. Let $\pi: Y\to X$ be a double cover branched at a very general element $B\in |M^{\otimes 2}|$.  Let $S$ be a very general element in $|H|$ and let $\tilde S=\pi^{-1}(S)$. Suppose there is a dominant rational map $f:\tilde S \dasharrow Z$ where $Z$ is a smooth projective surface. Then $Z$ is either birational to $\tilde S$, $S$, or we have $p_g(Z)=0.$
\end{corollary}

\begin{proof} Under our hypothesis, $H\otimes M^{-1}=\cO_{\bP^3}(d-a)$, $K_S=\cO_S(d-4)$, and $K_S+M=\cO_S(d+a-4)$ are very ample. And $h^0(\Omega^2_{\bP^3})=0$, hence it satisfies Assumption~\ref{xx}.
\end{proof}

\subsection{Rational maps}
By using deformation of curves and arguments similar to those in Section 2 in \cite{lp}, we can show :

\begin{theorem} \label{surface} Let $\pi: Y\to \bP^3$ be a double cover branched at a very general element in the linear system $|\cO_{\bP^3}(2)|$. Let $H=\cO_{\bP^3}(d)$ where $d\ge 7$ and let $S$ be a very general element in $|H|$. Set $\tilde S=\pi^{-1}(S).$
Suppose there is a dominant rational map $f:\tilde S \dasharrow Z$ where $Z$ is a smooth projective surface. Then $Z$ is either birational to $\tilde S$, $S$, or $\bP^2.$
\end{theorem}

\begin{proof} Suppose that $Z$ is not birational to $\tilde S$ and $S$. By Corollary~\ref{van}, it is enough to treat the case that $p_g(Z)=0$. We note that $Y$ is a quadric hypersurface in $\bP^4$. Let $D$ be a very general curve of $(d, d)$ type in a smooth quadric surface $Q$. Then we claim that there is no birational immersion $\kappa$ from $D$ into any smooth projective surface $Z$ with $p_g(Z)=q(Z)=0, \pi_1(Z)=1$, and non-negative Kodaira dimension if $d\ge 7$. The proof is similar to the argument in Section 2 in \cite{lp}.

We can assume that $Z$ is minimal because $D$ is a very general curve of $(d, d)$ type in $Q$.  Let $U$ be the Kuranishi space of deformation of $\kappa$. Since $\kappa$ is a very general birational immersion, a basic result (see Corollary 6.11 in \cite{ac}, and  Chapter XXI in \cite{ACG}) gives that
\[ \dim U\leq h^0(\cO_D(N_{D\mid Q}))-6=d^2+2d-6\]
by Riemann-Roch theorem and $h^1(\cO_D(N_{D\mid Q}))=0$.
And $g(D)=d^2-2d+1$.

Suppose that $D$ can be birationally immersed in $Z$ of general type with $p_g(Z) = q(Z) = 0, \pi_1(Z) = 1$. Then by the same argument in Proposition 2.3 in \cite{lp}
\[d^2+2d-6-19 \le g(D)-\frac{{\rm deg}\kappa^*(K_Z)}{2},\]
since minimal surfaces of general type with $p_g=q=0$ depends on 19 parameters (Corollary in \cite{gp}).
It implies that
\[4d-26+\frac{{\rm deg}\kappa^*(K_Z)}{2}\le 0.\]
So we get a contradiction if $d\ge 7$ because ${\rm deg}\,\kappa^*(K_Z) > 0$.

\medskip
Now, suppose $D$ can be birationally immersed in $Z$  with $p_g(Z) = q(Z) = 0, \pi_1(Z) = 1$, and of Kodaira dimension one.
Then
\[d^2+2d-6-10 \le g(D)-\frac{{\rm deg}\kappa^*(K_Z)}{2},\]
because minimal surfaces of Kodaira dimension one with $p_g(Z) = q(Z) = 0$, and $\pi_1(Z) = 1$ depend also on 10 parameters (cf. Proof of Proposition 3.5 in \cite{lp}).
It implies that
\[4d-17+\frac{{\rm deg}\kappa^*(K_Z)}{2}\le 0.\]
So we get a contradiction if $d\ge 5$ because ${\rm deg}\,\kappa^*(K_Z) \ge 0$. We prove the claim.

\medskip
Suppose there is a dominant rational map $p$ from $\tilde S$ to a smooth projective surface $Z$. Let $Z$ be a non-rational surface. By Corollary~\ref{van} we have $p_g(Z)=0$, and by the argument in \cite{gp} we have $\pi_1(Z)=1$. Since the intersection of $\tilde S$ and a very general hyperplane section of $Y$ is $D$,  we may assume that a general point of $Z$ belongs $f_D(D).$  By the above claim, $f_D$ cannot be birational. Therefore, we have two possible cases. The normalization of $f_D(D)$ is a very general plane curve of degree $d$ or rational. If the normalization of $f_D(D)$ is a very general plane curve of degree $d$ then we have a birational immersion from a very general plane curve of degree $d$ to $Z$. Then we get a contradiction by the result (proof of Theorem 1.1) in \cite{lp}.  If the normalization of $f_D(D)$ is rational then $Z$ is a ruled surface. It follows that $f$ is not dominant.
Therefore we get a contradiction.
\end{proof}


\end{document}